\documentclass[12pt]{article}
\usepackage{psfrag}
\usepackage{epsfig}
\usepackage{graphicx}
\usepackage{amsmath}
\usepackage{amssymb}
\usepackage{amsthm}
\usepackage{cite}
\usepackage[latin1]{inputenc}

\usepackage{amsfonts}
\usepackage{color}
\usepackage{amssymb,amsmath,amsthm,cite,enumerate}
\usepackage{mathrsfs}

\newtheorem{Theorem}[equation]{Theorem}

\newtheorem{Proposition}[equation]{Proposition}

\newtheorem{Lemma}[equation]{Lemma}

\newtheorem{Definition}[equation]{Definition}

\def\ef{\mathbb{F}}

\def\ov{\overline}
\def\sm{\setminus}

\begin{document}

\vspace*{5mm}

\noindent\textbf{\LARGE A combinatorial construction of an
$M_{12}$-invariant code
\footnote{This research was supported in part by Ministry for Education, University and Research of Italy (MIUR) and by the Italian National Group for Algebraic and Geometric Structures and their Applications (GNSAGA - INDAM).
}}

\flushbottom
\date{}

\vspace*{5mm} \noindent

\textsc{J\"urgen Bierbrauer} \\
\texttt{jbierbra@mtu.edu} \\
{\small Department of Mathematical Sciences Michigan Technological University Houghton, Michigan 49931 (USA)}\\

\textsc{Stefano Marcugini, Fernanda Pambianco} \\
\texttt{\{stefano.marcugini, fernanda.pambianco\}@unipg.it} \\
{\small Dipartimento di Matematica e Informatica,
Universit\`{a} degli Studi di Perugia, Via Vanvitelli~1,
Perugia, 06123, Italy}\\

\medskip

\begin{center}
\parbox{11,8cm}{\footnotesize
\textbf{Abstract.} In this work we summarized some recent results to be included in a forthcoming paper \cite{BMP}. A ternary $[66,10,36]_3$-code admitting the Mathieu group $M_{12}$ as a group
of automorphisms has recently been constructed by N. Pace, see \cite{Pace}.
We give a construction of the Pace code in terms of $M_{12}$ as well as a
combinatorial description in terms of the small Witt design, the Steiner system
$S(5,6,12).$ We also present a proof that the Pace code does indeed have
minimum distance $36.$}
\end{center}

\baselineskip=0.9\normalbaselineskip


\section{Introduction}
\label{introsection}
A large number of important mathematical objects are related to the Mathieu groups.
It came as a surprise when N. Pace found yet another such exceptional object, a
$[66,10,36]_3$-code whose group of automorphisms is $Z_2\times M_{12}$
(see \cite{Pace}). We present here two constructions for this code,
an algebraic construction which starts from the group $M_{12}$ in its natural action
as a group of permutations on $12$ letters, and
a combinatorial construction in terms of the Witt design $S(5,6,12).$
We also prove that the code has parameters as claimed.
In the next section we start by recalling some of the basic properties of
$M_{12}$ and the small Witt design $S(5,6,12).$

\section{The ternary Golay code, $M_{12}$ and $S(5,6,12)$}
\label{M12section}

 The Mathieu group $M_{12}$ is sharply 5-transitive on $12$ letters and therefore has order $12\times 11\times 10\times 9\times 8.$
It is best understood in terms of the ternary Golay code $[12,6,6]_3.$ The ternary Golay code has a generator matrix
$(I\vert P)$ where
$I$ is the $(6,6)$-unit matrix and $$P=\left(\begin{array}{cccccc}
0 &    1  &   1  & 1  & 1  &  1 \\
1 &    0  &   1  & 1  & 2  &  2 \\
1 &    1  &   0  & 2  & 1  &  2 \\
1 &    1  &   2  & 0  & 2  &  1 \\
1 &    2  &   1  & 2  & 0  &  1 \\
1 &    2  &   2  & 1  & 1  &  0
\end{array}\right).$$

The group $M_{12}$ acts in terms of monomial operations on the ternary Golay code.
Here we identify the $12$ letters with the columns of the generator matrix and consider the
action of $M_{12}$ as a group of permutations on those $12$ letters $\lbrace 1,2,\dots ,12\rbrace .$
It is generated by $h_1,h_2,h_3,h_4$ and $g$ where
$$h_1=(2,3,5,6,4)(8,9,11,12,10), h_2=(2,3)(4,5)(8,9)(10,11),$$
$$ h_3=(3,5,4,6)(9,11,10,12), h_4=(1,2)(5,6)(7,8)(11,12),$$
$$ g=(5,12)(6,11)(7,8)(9,10).$$
The group $H=\langle h_1,h_2,h_3,h_4\rangle$ of order $120$ is the stabilizer of $\lbrace 1,2,3,4,5,6\rbrace .$
  Call a 6-set an {\bf information set} if the
corresponding submatrix is invertible, call it a {\bf block} if the submatrix has rank $5.$
The terminology derives from the fact that the blocks define a Steiner system $S(5,6,12),$
the small Witt design. There are 132 blocks and $12\times 11\times 6$ information sets.
The complement of a block is a block as well.
The stabilizer of each 5-set is $S_5,$ the stabilizer of a block has order $10\times 9\times 8=720=6!$ and the
stabilizer of an information set has order $5!$ The stabilizer of a 2-set has order $1440.$ This stabilizer is the
group $P\Gamma L(2,9)\cong Aut(A_6).$
In the sequel we identify the $12$ letters with a basis $\lbrace v_1,\dots ,v_{12}\rbrace$ of a vector space
$V=V(12,3)$ over the field with three elements and consider the corresponding action of $M_{12}$
on $V.$

\section{The 10-dimensional module of $M_{12}$}
\label{10dimdef}

Clearly $M_{12}$ acts on an 11-dimensional submodule of $V,$ the
{\bf augmentation ideal} $I=\lbrace\sum_{i=1}^{12} a_iv_i\vert\sum a_i=0\rbrace $
and on a 1-dimensional submodule generated by the {\bf diagonal} $\Delta =v_1+\dots +v_{12}.$
As we are in characteristic 3, we have $\Delta\in I,$ and $M_{12}$ acts on the 10-dimensional factor space
$Z=I/\langle\Delta\rangle .$
The $u_i=v_i-v_{12}, i\leq 11$ are a basis of $I$ and $z_i=\ov{u_i}=u_i+\Delta\ef_3, i\leq 10$ are a basis of $Z.$
Here $\sum_{i=1}^{11} u_i=\Delta ,$ hence $z_{11}=-z_1-\dots -z_{10}.$

\section{The Pace code}
\label{Pacealgdef}

We consider the action of $M_{12}$ on the 10-dimensional $\ef_3$-vector space $Z$ with its basis
$z_i=\ov{u_i}=v_i-v_{12}+\Delta\ef_3, i=1,\dots ,10.$ Recall that it is induced by the
permutation representation on $\lbrace v_1,\dots ,v_{12}\rbrace .$
This action defines embeddings of  $M_{12}$ in $GL(10,3)$ and in $PGL(10,3).$
For each orbit of $M_{12}$ we consider the projective ternary code whose generator matrix has as columns
representatives of the projective points constituting the orbit.

\begin{Definition}
\label{Paceblockdef}
Let $X\subset\lbrace 1,2,\dots ,12\rbrace , \vert X\vert =6.$\\
Define $v_X=\sum_{i\in X}v_i, z_X=\ov{v_X}.$
\end{Definition}

It is in fact clear that $v_X\in I,$ and $z_X\in Z$ is therefore defined.

\begin{Proposition}
\label{Paceorbitprop}
The $z_X\in Z$ where $X$ varies over the blocks of $S(5,6,12)$
form an orbit of length $132$ in $Z.$
In the action on projective points (in $PG(9,3)$), this yields an orbit
of length $66.$
\end{Proposition}

\begin{proof}
Clearly $M_{12}$ permutes the $z_X$ in the same way as it permutes the
blocks $X.$ This yields an orbit of length $132$ in $Z=V(10,3).$
If $\ov{X}$ is the complement of $X,$ then $v_{\ov{X}}+v_X=\Delta ,$ hence
$z_{\ov{X}}=-z_X.$
It follows that  $M_{12}$ acts transitively on the $66$ points in $PG(9,3)$ generated by the $z_X$
(block $X$ and its complement generating the same projective point).
\end{proof}

\begin{Definition}
\label{PaceM12def}
Let $C$ be the $[66,10]_3$-code whose generator matrix has as columns
representatives of the orbit of $M_{12}$ on the $z_X$ where $X$ is a block.
\end{Definition}

This is one way of representing the Pace code. Observe that each complementary pair of
blocks contributes one column of the generator matrix. We may use as representatives the
vectors $z_X$ where $X$ varies over the $66$
blocks $X$ not containing the letter $12.$
As the stabilizer of a block in $M_{12}$ is $S_6$ it follows that the stabilizer of a point in the orbit
equals the stabilizer of a complementary pair of blocks and is twice as large as $S_6.$
The stabilizer is $P\Gamma L(2,9),$ of order $2\times 6!$

\section{A combinatorial description}
\label{combdescsection}

We introduce some notation.

\begin{Definition}
\label{combdescdef}
Let ${\cal B}$ be a family of subsets (blocks) of a $v$-element set $\Omega .$
Let $A,B\subset\Omega$ be disjoint subsets, $\vert A\vert =a, \vert B\vert =b.$
Define a matrix $G$ with $k=v-a-b$ rows and $n$ columns where $n$ is the number of blocks
disjoint from $A.$
Here we identify the rows of $G$ with the points $i\in\Omega\sm (A\cup B)$ and the
columns with the blocks $X$ disjoint from $A.$
The entry in row $i$ and column $X$ is $=1$ if $i\in X,$ it is $=0$ otherwise.
As the entries of $G$ are $0,1$ we can consider them as
elements of an arbitrary finite field $K.$
Define ${\cal C}=C_{A,B}({\cal B},K)$ to be the code generated by $G$ over $K.$
\end{Definition}

In words: the column of $G$ indexed by $X\in {\cal B}$ is the characteristic function of the
set $X\sm B.$
We write $C_{a,b}({\cal B},K)$ instead if the choice of the subsets $A,B$ does not matter.
This is the case in particular if the automorphism group of ${\cal B}$ is $(a+b)$-transitive.
Code ${\cal C}$ is a $K$-linear code of length $n.$ Its designed dimension is $k$
but the true dimension may be smaller.
We have no clue what the minimum distance is.
Observe that $C_{A,B}({\cal B},K)$ is a subcode of $C_{A,\emptyset }({\cal B},K):$
a generator matrix of the smaller code arises from the generator matrix of the larger code by
omitting some $\vert B\vert$ rows.

\begin{Proposition}
\label{PAXidentprop}
The Pace code from Definition \ref{PaceM12def} is monomially equivalent to
$C_{1,1}(S(5,6,12),\ef_3).$
\end{Proposition}
\begin{proof}
The generator matrix of Definition \ref{PaceM12def} has rows indexed by
$i\in\lbrace 1,\dots ,10\rbrace$ and columns indexed by blocks $X$ of $S(5,6,12)$
not containing the letter $12.$ If also $11\notin X,$ then the corresponding column is
the characteristic function of $X.$ Let $11\in X.$ As $z_{11}=-z_1-\dots -z_{10}$ the
entries in this column are $=0$ if $i\in X, =2$ if $i\notin X.$ Taking the negative of this
column, we obtain the characteristic function of $\ov{X}\sm\lbrace 12\rbrace .$
We arrive at the generator matrix of $C_{A,B}(S(5,6,12),\ef_3)$
where $A=\lbrace 11\rbrace , B=\lbrace 12\rbrace .$
\end{proof}

\section{Combinatorial properties of the small Witt design}
\label{comblemmasection}

The following elementary properties of the Steiner system $S(5,6,12)$ will be used in the sequel.

\begin{Lemma}
\label{iablemma}
Let $\Omega =\lbrace 1,2,\dots ,12\rbrace$ and $A,B\subset\Omega , \vert A\vert =a, \vert B\vert =b$
and such that $A\cap B=\emptyset ,a+b\leq 5.$ Let $i(a,b)$ be the number of blocks which
contain $A$ and are disjoint from $B.$
Then $i(b,a)=i(a,b)$ and
$$i(5,0)=1, i(4,0)=4, i(3,0)=12, i(2,0)=30, i(1,0)=66,$$
$$i(1,1)=36, i(2,1)=18, i(3,1)=8, i(4,1)=3, i(2,2)=10, i(3,2)=5.$$
\end{Lemma}
\begin{proof}
$i(5,0)=1$ is the definition of a Steiner $5$-design, $i(b,a)=i(a,b)$ follows from the fact that
the complements of blocks are blocks. The rest follows from obvious counting arguments.
\end{proof}

The following combinatorial lemmas may be verified by direct calculations using coordinates.

\begin{Lemma}
\label{specialcountthirdlemma}
A family of five 3-subsets of a 6-set contains at least two 3-subsets which meet in 2 points.

\end{Lemma}

\begin{Lemma}
\label{specialcountfirstlemma}
Let $U\subset\lbrace 1,2,\dots, 11\rbrace$ such that $\vert U\vert =6.$
The number of blocks $B\in {\cal B}$ such that
$\vert B\cap U\vert =3$ is $20$ if $U$ is a block, it is $30$ if $U$ is not a block.
\end{Lemma}

\begin{Lemma}
\label{k2443lemma}
Let $\Omega =\lbrace 1,2,\dots ,12\rbrace$ and $\Omega =A\cup B\cup C$
where $\vert A\vert =\vert B\vert =\vert C\vert =4$ and $P\in C.$
The number of blocks which meet each of $A,B,C$ in cardinality $2$ and avoid $P$
is at most $18.$
\end{Lemma}


\section{The parameters of the Pace code}
\label{paramsection}

\begin{Theorem}
\label{PAXtheorem}
The Pace code is a self-orthogonal $[66,10,36]_3$-code.
\end{Theorem}

In the remainder of this section we prove Theorem \ref{PAXtheorem}.
We use the Pace code in the form $C=C_{A,B}(S(5,6,12),\ef_3)$
where $A=\lbrace 12\rbrace , B=\lbrace 11\rbrace ,$ see Definition \ref{combdescdef}.
The length is $n=i(0,1)=66,$ the designed dimension is $k=10.$
Let ${\cal B}$ be the blocks of $S(5,6,12)$ not containing $12.$ Observe that the columns
of $G$ are the characteristic functions of $X\sm\lbrace 11\rbrace$ where $X\in {\cal B}.$
Let $r_i, 1\leq i\leq 10$ be the rows of the generator matrix of Definition \ref{combdescdef}.
The codewords of $C$ have the form $\sum_{i\in U}r_i-\sum_{j\in V}r_j,$ where $U,V$ are
disjoint subsets of $\lbrace 1,\dots ,10\rbrace .$
The number of zeroes of this codeword is
the  {\bf nullity} $\nu (U,V),$  the number of blocks $X\in {\cal B}$ satisfying the condition that
$\vert X\cap U\vert$ and $\vert X\cap V\vert$ have the same congruence mod $3.$
Let $c\in\lbrace 0,1,2\rbrace$ be this congruence.
We need to show that $\nu (U,V)\leq 30$ for all $(U,V)$ except when $U=V=\emptyset .$
This will prove the claim that the nonzero weights are $\geq 36$ and also that
the dimension is $10.$
\par
Let $u=\vert U\vert , v=\vert V\vert ,$ let $W$ be the complement of $U\cup V$ in
$\lbrace 1,\dots ,11\rbrace , w=\vert W\vert .$ Observe $u+v+w=11, w>0.$
We have $\nu (U,V)=\sum_ck_c(u,v,w),$ where $k_c(u,v,w)$ is the number of $X\in {\cal B}$
meeting each of $U,V,W$ in a cardinality congruent to $c$ mod 3.
Observe that $k_c(u,v,w)$ is symmetric in its arguments
as long as the side condition $w>0$ is satisfied.
The weight of $r_i$ is $i(1,1)=36$ (this is case $u=1,v=0$).
In particular $r_i\cdot r_i=0.$ Also $r_i\cdot r_j=0$ for $i\not=j$ as $i(2,1)=18$ is a multiple of $3.$
It follows that $C$ is self-orthogonal. All codeword weights and nullities are therefore multiples of $3.$
 \par
It may be verified that $\nu (U,V)<33$ in a case by case analysis, starting from large values of $u.$
If $u=10$ then $v=0,w=1$ and $\nu (10,0)=k_0(10,0,1)=i(0,2)=30.$
Cases $u\in\lbrace 6,7,8,9\rbrace$are similar.

Let $u=5.$ By symmetry it can be assumed $3\leq v\leq 5.$
In case $v=5$ we have $k_1(5,5,1)\leq 10, k_0(5,5,1)\leq 20,$
and in case $v=4$ we have $k_2(5,4,2)\leq 12, k_1(5,4,2)\leq 2+10=12, k_0(5,4,2)\leq 8,$ hence
$\nu (5,4)<33.$
As $k_2(5,3,3)\leq 18, k_1(5,3,3)\leq 9$ and $k_0(5,3,3)\leq 1+2\times 2$ we have $\nu (5,3)<33.$
The final case to consider is $(u,v,w)=(4,4,3).$
In case $c=0$ we have that $X$ meets two of the subsets $U,V,W$ in cardinality $3.$
If $W\subset X,$ there are at most two such blocks.
There are at most four blocks meeting each of $U,V$ in cardinality $3.$ It follows $k_0(4,4,3)\leq 6.$
If $c=1,$ then either $U\subset X$ or $V\subset X.$ It follows $k_1(4,4,3)\leq 6.$
The most difficult case is $c=2.$ Lemma \ref{k2443lemma} states $k_2(4,4,3)\leq 18.$
We are done.

\end{document}